\documentclass[11pt]{article}                                                                                                       
\usepackage{epsfig,psfrag,amsmath,amssymb,latexsym}

\usepackage{amsmath}
\usepackage{amssymb}
\usepackage{amsfonts}
\usepackage{latexsym}
\usepackage{amsthm}
\usepackage{amsxtra}
\usepackage{amscd}
\usepackage{mathrsfs}
\usepackage{mathscinet}
\usepackage{bm}
\usepackage{graphicx}

\usepackage{color}
\usepackage{cite}
\usepackage{hyperref}
\usepackage{amscd}
\usepackage{color}
\usepackage{amsfonts}
\usepackage{graphicx}
\usepackage{wasysym}
\usepackage{mathrsfs}
\usepackage{bbm}
\theoremstyle{plain} 
\newtheorem{theorem}{Theorem}
\newtheorem*{theorem*}{Theorem}
\newtheorem{lemma}[theorem]{Lemma}
\newtheorem*{lemma*}{Lemma}
\newtheorem{corollary}[theorem]{Corollary}
\newtheorem*{corollary*}{Corollary}
\newtheorem{fact}[theorem]{Fact}
\newtheorem{proposition}[theorem]{Proposition}
\newtheorem*{proposition*}{Proposition}

\newtheorem*{definition*}{Definition}

\newtheorem*{conjecture*}{Conjecture}

\newtheorem*{example*}{Example}

\newtheorem*{remark*}{Remark}

\newcommand{\Lstar}{\Lambda^*}
\newcommand{\OLstar}{\Omega_\Lambda^*}
\newcommand{\obs}{\mathcal{O}}

\newcommand{\var}{\operatorname{Var}}

\newcommand{\K}{\mathcal{K}}
\newcommand{\cE}{\mathcal{E}}

\newcommand{\cL}{\mathcal{L}}

\newcommand{\ee}{{\mathrm e}}
\newcommand{\ii}{{\mathrm i}}  

\newcommand{\R}{{\mathbb R}}  
\newcommand{\N}{{\mathbb N}}  
  
\newcommand{\Z}{{\mathbb Z}}

\newcommand{\md}{{\mathrm{d}}}

\usepackage[dvipsnames]{xcolor}

\begin{document}
\thispagestyle{empty}
\begin{center}
{\bf\Large Variance of voltages in a lattice Coulomb gas}
\\[3mm]
Diana Conache\footnote[1]{%
Technische Universit{\"{a}}t M{\"{u}}nchen, 
Zentrum Mathematik, Bereich M5,
D-85747 Garching bei M{\"{u}}nchen,
Germany.
E-mail: diana.conache@tum.de, srolles@ma.tum.de}
\hspace{1mm} 
Markus Heydenreich\footnotemark[2]
\hspace{1mm} 
Franz Merkl\footnote[2]{Mathematical Institute, Ludwig-Maximilians-Universit\"at M\"unchen,
Theresienstr.\ 39,
D-80333 Munich,
Germany.
E-mail: m.heydenreich@lmu.de, merkl@math.lmu.de
}
\hspace{1mm} 
Silke W.W.\ Rolles\footnotemark[1]
\\[3mm]
{\small \today}\\[3mm]
\end{center}
\begin{abstract}
We study the behavior of the variance of the difference of energies for putting an additional electric unit charge at two different locations in the two-dimensional lattice Coulomb gas in the high-temperature regime. For this, we exploit the duality between this model and a discrete Gaussian model. Our estimates follow from a spontaneous symmetry breaking in the latter model. \\[2mm]
\end{abstract}

\section{Introduction}
The Poisson summation formula is by now a standard tool in Statistical Mechanics for proving duality between models in the high and low-temperature
regime. It was first used by McKean \cite{MR162515} to prove the Kramers-Wannier duality for the infinite two-dimensional Ising Model. Further
results for other models followed. Gruber and Hintermann \cite{gruber1974implication} showed that for arbitrary lattice spin systems one can relate the high and low temperature expansions by means of the Poisson summation formula. Several overview papers of the results on duality appeared in the beginning of the 1980's both from a mathematical perspective \cite{MR625131} and from a physics one \cite{MR569167}.

In this note, we make use of the Poisson summation formula to show the duality between a two-dimensional discrete Gaussian model with pinning at the origin at low temperature and a lattice Coulomb gas at high temperature. Such lattice Coulomb gas representations were previously treated in \cite{Kadanoff_1978}. Expected squared height differences in the discrete Gaussian model correspond by duality to variances of voltages in the lattice Coulomb gas. These voltages are highly non-local observables in terms of the charge distribution, so that known results on Debye screening \cite{B, BF, BM} do not apply to them. In Corollary \ref{cor:symmetry_breaking} we give a proof of the spontaneous breaking of the internal $\Z^1$ symmetry in the discrete Gaussian model via a Peierls argument on the torus, taking also care of contours with non-zero homology. This comes in contrast to a continuous version of the model, where the symmetry cannot be broken, even for more general models due to the Mermin-Wagner theorem. We then employ this symmetry breaking to show that the variance of the energy cost for transporting a unit test charge between two given sites $i$ and $j$ of a lattice, i.e.\ the variance of the voltage between $i$ and $j$,  is asymptotically proportional to $(\beta^*)^{-1}\log|i-j|$ for small inverse temperature $\beta^*$. We make this precise in Theorem ~\ref{thm:main} below. 

The duality between the discrete Gaussian model and the lattice Coulomb gas was also discussed by Fr\"ohlich and Spencer in \cite{FS}; see also earlier work on the Kosterlitz-Thouless transition by \cite{KT, Villain}. A discussion on the connection to the results in \cite{FS} is included in the last section. There we also contrast our results to Debye screening proven in three space dimensions by Brydges and Federbush \cite{B, BF}.

\section{The two models}
\subsection{The discrete Gaussian model}
Let $\Lambda\Subset \Z^2$ be a finite square box with periodic boundary conditions containing the origin and let 
$\cE_\Lambda$ denote the set of undirected edges between nearest-neighbor points in $\Lambda$. 
Consider the space of integer-valued configurations which are pinned at the origin:
\begin{align}
\Omega_\Lambda=\{x\in\Z^\Lambda:x_0=0\}. 
\end{align}
We consider the Hamiltonian 
\begin{align}
H_\Lambda(x)=\sum_{\{i,j\}\in\cE_\Lambda}(x_i-x_j)^2, \quad x\in\Omega_\Lambda.
\end{align}
Using the lattice Laplacian $\Delta=(\Delta_{ij})_{i,j\in\Lambda}$, 
\begin{align}
\Delta_{ij}=\left\{\begin{array}{ll}
1 & \text{if }i\sim j, \\
-4 & \text{if }i=j,\\
0 & \text{else, }
\end{array}\right.
\end{align}
we can rewrite it as 
\begin{align}
H_\Lambda(x)=-\sum_{i,j\in\Lambda}x_i\Delta_{ij}x_j,\quad x\in\Omega_\Lambda.
\end{align}
The partition sum and the corresponding Gibbs measure of 
the discrete Gaussian model with pinning at the origin at given inverse temperature 
$\beta>0$ are given by 
\begin{align}
Z_{\Lambda,\beta}=\sum_{x\in\Omega_\Lambda}\ee^{-\beta H_\Lambda(x)} \quad \text{ and } \quad
P_{\Lambda,\beta}=\frac{1}{Z_{\Lambda, \beta}}\sum_{x\in\Omega_\Lambda}\ee^{-\beta H_\Lambda(x)}\delta_x, 
\label{def:finite-volume-Gibbs-measure}
\end{align}
respectively; here $\delta_x$ denotes the Dirac measure in $x$.

\subsection{The lattice Coulomb gas}

Consider the Fourier Transform of a given 
integrable function $f$ on $\R^n$: 
\begin{align}
\hat{f}(k)=\int_{\R^n} \ee^{-\ii k\cdot x}f(x)\md x, \qquad k\in\R^n. 
\end{align}

\paragraph{Poisson Summation Formula.} We recall a well-known duality result from Fourier Analysis.

\begin{fact}[Poisson Summation Formula, {\cite[Theorem 3.2.8]{MR3243734}}]
\label{fact:PoissonSummationFormula}
Let $n\in\N$ and let $f$ be a continuous function on $\R^n$ which satisfies 
for some $C, \delta>0$ and for all $x\in\R^n$
\begin{align}
|f(x)|\le C(1+|x|)^{-n-\delta},
\end{align}
and whose Fourier transform $\hat{f}$ restricted to $2\pi\Z^n$ satisfies
\begin{align}
\label{eq:PoissonSummationFormula-condition}
\sum_{k\in 2\pi\Z^n}|\hat{f}(k)|<\infty.
\end{align}
Then for all $y\in\R^n$ we have
\begin{align}
\label{eq:PoissonSummationFormula}
\sum_{k\in2\pi\Z^n} \hat{f}(k)\ee^{\ii k\cdot y}=\sum_{x\in\Z^n}f(x+y),
\end{align}
and in particular
\begin{align}
\label{eq:PoissonSummationFormula_simple}
\sum_{k\in2\pi\Z^n} \hat{f}(k)=\sum_{x\in\Z^n}f(x).
\end{align}
\end{fact}
In order to show the duality between the discrete Gaussian model and the lattice Coulomb gas, we make use of the Poisson summation formula by applying it to the Boltzmann factor of the first model. Note that, for a positive definite matrix $A\in\R^{n\times n}$ and $g(x)=\exp(-x^tAx)$, one has 
\begin{align}\label{eq:Fourier-g}
\hat g(k)=\frac{\pi^{\frac{n}{2}}}{\sqrt{\det A}}\ee^{-\frac14k^t A^{-1} k}.
\end{align}
Since $g$ satisfies the hypotheses of Fact \ref{fact:PoissonSummationFormula}, 
formulas \eqref{eq:PoissonSummationFormula} - \eqref{eq:PoissonSummationFormula_simple} 
hold for $g$.

\paragraph{Green's function.}
Let $P_i$ be the law of a simple random walk $(X_t)_{t\in\N_0}$
on $\Lambda$ started at $i\in\Lambda$. 
For $i,j\in\Lambda$, we define the Green's function 
\begin{align}
G_{ij}=\lim_{\lambda\downarrow 0}
\sum_{t\in\N_0}\left(P_i(X_t=j)-\frac{1}{|\Lambda|}
\right)\ee^{-\lambda t}. 
\end{align}
The limit exists and is finite. To see this, combine odd and even times 
in pairs of two and use that $P_i(X_{2t}=j)+P_i(X_{2t+1}=j)$ converges exponentially 
fast to $2/|\Lambda|$ as $t\to\infty$. 

\begin{lemma}
\label{le:sum-delta-G}
For $j,k\in\Lambda$, we have 
\begin{align}
-\frac14\sum_{i\in\Lambda}\Delta_{ki}G_{ij}=\delta_{kj}-\frac{1}{|\Lambda|}. 
\end{align}
\end{lemma}
\begin{proof}
For all $k,i\in\Lambda$, we have
\begin{align}
\Delta_{ki}=4(P_k(X_1=i)-\delta_{ki}). 
\end{align}
Therefore, for $j,k\in\Lambda$, it follows
\begin{align}
& -\frac14\sum_{i\in\Lambda}\Delta_{ki}G_{ij}
=-\sum_{i\in\Lambda}(P_k(X_1=i)-\delta_{ki})G_{ij} 
=G_{kj} -\sum_{i\in\Lambda}P_k(X_1=i)G_{ij}
\nonumber\\
= & \lim_{\lambda\downarrow 0}
\sum_{t\in\N_0}\left(P_k(X_t=j)-\frac{1}{|\Lambda|}
\right)\ee^{-\lambda t}- \sum_{i\in\Lambda}P_k(X_1=i)
\lim_{\lambda\downarrow 0}
\sum_{t\in\N_0}\left(P_i(X_t=j)-\frac{1}{|\Lambda|}
\right)\ee^{-\lambda t}
\nonumber\\
= & \lim_{\lambda\downarrow 0}
\sum_{t\in\N_0}\left(P_k(X_t=j)-\frac{1}{|\Lambda|}
\right)\ee^{-\lambda t}- 
\lim_{\lambda\downarrow 0}
\sum_{t\in\N_0}\left(P_k(X_{t+1}=j)-\frac{1}{|\Lambda|}
\right)\ee^{-\lambda t}\ee^{-\lambda}
\nonumber\\
=&P_k(X_0=j)-\frac{1}{|\Lambda|}
= \delta_{kj}-\frac{1}{|\Lambda|}. 
\end{align}
For the third line we used the Markov property and introduced an 
extra factor $\ee^{-\lambda}$, which converges to $1$ as $\lambda\downarrow0$. 
\end{proof}

The following identity also has a continuous analogue for the Green's function of a randomly shifted Gaussian Free Field. More precisely, see Remark 8.20 and Exercise 8.7 in \cite{friedli-velenik2018} for the continuum Gaussian Free Field and Exercise 1.4. in \cite{MR4043225} for the discrete Gaussian Free Field.
\begin{lemma}
\label{lem:inverse_hamiltonian}
For $i,j\in\Lambda\setminus\{0\}$, we have that the inverse of the matrix $\Delta_{0^c0^c}=(\Delta_{ij})_{i,j\in\Lambda\setminus\{0\}}$ is given by
\begin{align}
-4((\Delta_{0^c0^c})^{-1})_{ij}=G_{ij}-G_{i0}-G_{0j}+G_{00}.
\end{align}
\end{lemma}

\begin{proof}
Since 
\begin{align}
\sum_{i\in\Lambda}\Delta_{ki}=0
\end{align}
for all $k\in\Lambda$, we calculate for $j,k\in\Lambda\setminus\{0\}$, 
using Lemma \ref{le:sum-delta-G} in the third step,  
\begin{align}
&\sum_{i\in\Lambda\setminus\{0\}}\Delta_{ki}(G_{ij}-G_{i0}-G_{0j}+G_{00})
=\sum_{i\in\Lambda}\Delta_{ki}(G_{ij}-G_{i0}-G_{0j}+G_{00})\nonumber\\
=& \sum_{i\in\Lambda}\Delta_{ki}(G_{ij}-G_{i0})
=-4\left(\delta_{kj}-\frac{1}{|\Lambda|}\right)+4\left(\delta_{k0}-\frac{1}{|\Lambda|}\right)
=-4\delta_{kj}.
\end{align}
\end{proof}

\paragraph{Duality between the discrete Gaussian model and the lattice Coulomb gas.}
The matrix $\Delta_{0^c0^c}=(\Delta_{ij})_{i,j\in\Lambda\setminus\{0\}}$ is 
negative definite. Applying the Poisson summation formula \eqref{eq:PoissonSummationFormula_simple} to $g(x)=\exp(-x^tAx)$ with 
$n=|\Lambda|-1$ and $A=-\beta \Delta_{0^c0^c}$, we connect the partition sum $Z_{\Lambda,\beta}$ of the discrete Gaussian model defined in \eqref{def:finite-volume-Gibbs-measure} with the partition sum
\begin{equation}\label{eq:Z-stern}
Z^*_{\Lambda,\beta^{*}}:=\sum_{k\in 2\pi\Z^{\Lambda\setminus\{0\}}}
\ee^{\beta^* k^t(\Delta_{0^c0^c})^{-1}k}
\end{equation}
of the lattice Coulomb gas at inverse temperature $\beta^{*}:=(4\beta)^{-1}$ by the following duality equation
\begin{align}
\label{eq:Z-poisson}
Z_{\Lambda,\beta}=\sum_{x\in\Z^{\Lambda\setminus\{0\}}}\ee^{-\beta x^t(-\Delta_{0^c0^c})x}
=\left(\frac{\pi}{\beta}\right)^{\frac{|\Lambda|-1}{2}}
\det(-\Delta_{0^c0^c})^{-\frac12}Z^*_{\Lambda,\beta^{*}}. 
\end{align}
Let $\K_\Lambda:=\{k\in\R^\Lambda:\,\sum_{i\in\Lambda} k_i=0\}$, which we view as dual to $\Omega_\Lambda^\R:=\{x\in\R^\Lambda:x_0=0\}$. Define  $\OLstar:=2\pi\Z^{\Lambda}\cap \K_\Lambda$. The Gibbs measure of the lattice Coulomb gas on $\OLstar$ is given by
\begin{equation}\label{def:GibbsCG}
P_{\Lambda,\beta^{*}}^*:=\frac{1}{Z_{\Lambda, \beta^{*}}^*}\sum_{k\in 2\pi\Z^{\Lambda\setminus\{0\}}}
\ee^{\beta^* k^t(\Delta_{0^c0^c})^{-1}k}\delta_k. 
\end{equation}

\begin{corollary}[The dual lattice Coulomb gas in terms of the Green's function]
\label{cor:rel-G-delta-inverse}
For all $k\in\K_\Lambda$, writing $k_{0^c}=(k_i)_{i\in\Lambda\setminus\{0\}}$, one has 
\begin{align}
-4k_{0^c}^t(\Delta_{0^c0^c})^{-1}k_{0^c}=k^tG k.
\end{align}
In particular, 
\begin{align}\label{eq:Coulomb-partition-function}
Z^*_{\Lambda,\beta^{*}}=\sum_{k\in \OLstar}\ee^{-\frac{\beta^*}{4} k^tG k} \quad \text{ and } \quad P_{\Lambda,\beta^{*}}^*=\frac{1}{Z_{\Lambda, \beta^{*}}^*}\sum_{k\in\OLstar}\ee^{-\frac{\beta^*}{4} k^tG k}\delta_k. 
\end{align}
\end{corollary}

\begin{proof}
Let $\mathbf{1}=(1)_{i\in\Lambda\setminus\{0\}}$ denote the column vector consisting of 
ones. For $k\in\K_\Lambda$ we have by Lemma \ref{lem:inverse_hamiltonian}
\begin{align}
k^tG k&=
k_{0^c}^tG_{0^c0^c} k_{0^c}+
k_{0^c}^tG_{0^c0} k_0+k_0G_{00^c} k_{0^c}+k_0G_{00} k_0
\nonumber\\&
=
k_{0^c}^t(G_{0^c0^c}
-G_{0^c0}\mathbf{1}^t-\mathbf{1}G_{00^c}+\mathbf{1}G_{00}\mathbf{1}^t) k_{0^c}
=-4k_{0^c}^t(\Delta_{0^c0^c})^{-1}k_{0^c}.
\end{align}
Substituting this in \eqref{eq:Z-stern} and \eqref{def:GibbsCG}, the second claim also follows.
\end{proof}

\section{Peierls argument and breaking of symmetry in the discrete Gaussian model}
The following results follow the lines of Section 6.3 in \cite{MR2807681}, where a thorough discussion 
regarding the symmetries and the ground states of the discrete Gaussian model (without pinning) based on \cite{MR691041} is given.  While Peierls arguments on planar domains are classical, the treatment of non-zero homologies are less common.

We construct a contour model as follows. Let $\Lstar:=\Lambda+(1/2, 1/2)$ be the dual box with periodic boundary conditions and $\cE_{\Lstar}$ the set of edges between nearest-neighbor points in $\Lstar$.  Let $x\in\Omega_{\Lambda}$ 
and $\{i,j\}\in\cE_\Lambda$. There exists a unique $\{i,j\}^*\in\cE_{\Lstar}$ intersecting $\{i,j\}$. Then, if $x_i-x_j=n\in\N$, we draw $n$ distinct arrows on $\{i,j\}^*$ such that when one looks in the direction of the arrows, the vertex $i$ is to the left and the vertex $j$ to the right. In other words, the larger value is to the left of the arrow and the smaller one to the right. For a dual edge $e$ directed from dual vertex $e_-$ to $e_+$, we denote its direction vector by $\vec{e}:=e_+-e_-$.

Take now two different vertices $i, j\in\Lambda$ and fix a configuration $x$ in the set $\Omega_{i,j}:=\{x\in\Omega_\Lambda: x_i>x_j\}$. Consider the set $M_i:=\{k\in\Lambda: x_k\geq x_i\}$. Let $C_i$ be the connected component of $M_i$ containing $i$. Let $\partial C_i$ be its boundary, viewed as a set of directed edges in the dual lattice, where the direction follows the same rule as for the arrows described above. The connected components of $\partial C_i$ consist of closed contours; at the intersection of four dual edges belonging to contours we use the North-West/South-East deformation rule (see e.g. Fig. 3.11 in \cite{friedli-velenik2018}) to generate self-avoiding contours. 

Zero or two of the connected components of $\partial C_i$ may wind around the torus (let's call them $\gamma_{-}$ and $\gamma_{+}$, if they exist), but all other contours $\gamma_1, \ldots, \gamma_m$ do not wind around it. One can see this as follows. Winding connected components $\gamma$ are characterized by their period $p(\gamma):=\sum_{e\in\gamma}\vec{e}$ being non-zero. The period characterizes the homology class of $\gamma$. Due to the periodic boundary conditions, for any set of vertices $C\subseteq\Lambda$, the balance equation $\sum_{e\in\partial C}\vec{e}=0$ holds. This shows that there are either zero or at least two winding connected components of $\partial C_i$.  The period $p(\gamma)$ is the asymptotic direction of any lifting $\hat{\gamma}$ of $\gamma$ to a two-sided infinite path in $\Z^2$. Here a lifting $\hat{\gamma}$ means a connected component of the inverse image of $\gamma$ w.r.t.\ the canonical map $\Z^2+(1/2, 1/2)\to\Lstar$. Because the liftings $\hat{\gamma}$ and $\hat{\gamma}'$ of any two different connected components $\gamma$ and $\gamma'$ of $\partial C_i$ are disjoint, their periods $p(\gamma), p(\gamma')$ are linearly dependent. Assume that we have at least three different infinite connected components $\hat{\gamma}, \hat{\gamma}', \hat{\gamma}''$ in the boundary $\partial \hat{C_i}$ of a lifting $\hat{C_i}$ of $C_i$. Since their asymptotic directions are pairwise linearly dependent, two of them are separated by the third, a contradiction. 

Therefore there are two cases, illustrated in Figure \ref{figure}:
\begin{itemize}
\item \textit{Case 1, connected cycles: }either precisely one of the contours $\gamma_1, \ldots, \gamma_m$, called $\gamma_{i,j}$, separates $i$ from $j$, or
\item \textit{Case 2, winding pairs: }$\gamma_{-}$ and $\gamma_{+}$ exist and their union $\gamma_{i,j}:=\gamma_{-}\cup \gamma_{+}$ separates $i$ and $j$.
\end{itemize} 

Let $\cL_{i}(\gamma_{i,j})$ denote the set of all vertices $k\in\Lambda$ such that $\gamma_{i,j}$ does not separate $i$ from $k$. Note that $\cL_{i}(\gamma_{i,j}(x))$ depends on the configuration $x$ only through $\gamma_{i,j}(x)$.
Consider the mapping 
\begin{align}
F_{i,j}:\Omega_{i,j}\longrightarrow \Omega_{\Lambda}, \qquad
x\longmapsto F_{i,j}(x),
\end{align}
where $F_{i,j}(x)_l=x_l'-x_0'$, with $x_l'=x_l-1$ for $k\in\cL_{i}(\gamma_{i,j}(x))$ and $x_l'=x_l$, otherwise. In other words, the map $F_{i,j}$ lowers the configuration $x$ in the region bounded by the oriented contour(s) $\gamma_{i,j}$. In compact notation, $F_{i,j}(x)_l=x_l-\mathbf{1}_{\cL_{i}(\gamma_{i,j}(x))}(l)-x_0+\textbf{1}_{\cL_{i}(\gamma_{i,j}(x))}(0)$. Note that for any edge $\{l,m\}\in\cE_\Lambda$ one has 
\begin{align}\label{eq:F_ij}
|F_{i,j}(x)_l-F_{i,j}(x)_m|=|x_l-x_m|-\mathbf{1}_{\{\{l,m\}^*\in\gamma_{i,j}(x) \}}.
\end{align}
Here, the event ${\{\{l,m\}^*\in\gamma_{i,j} \}}$ means that one dual edge in $\gamma_{i,j}$ intersects the edge $\{l,m\}$.

\begin{figure}[ht]
\centering\includegraphics[width=0.6\textwidth] {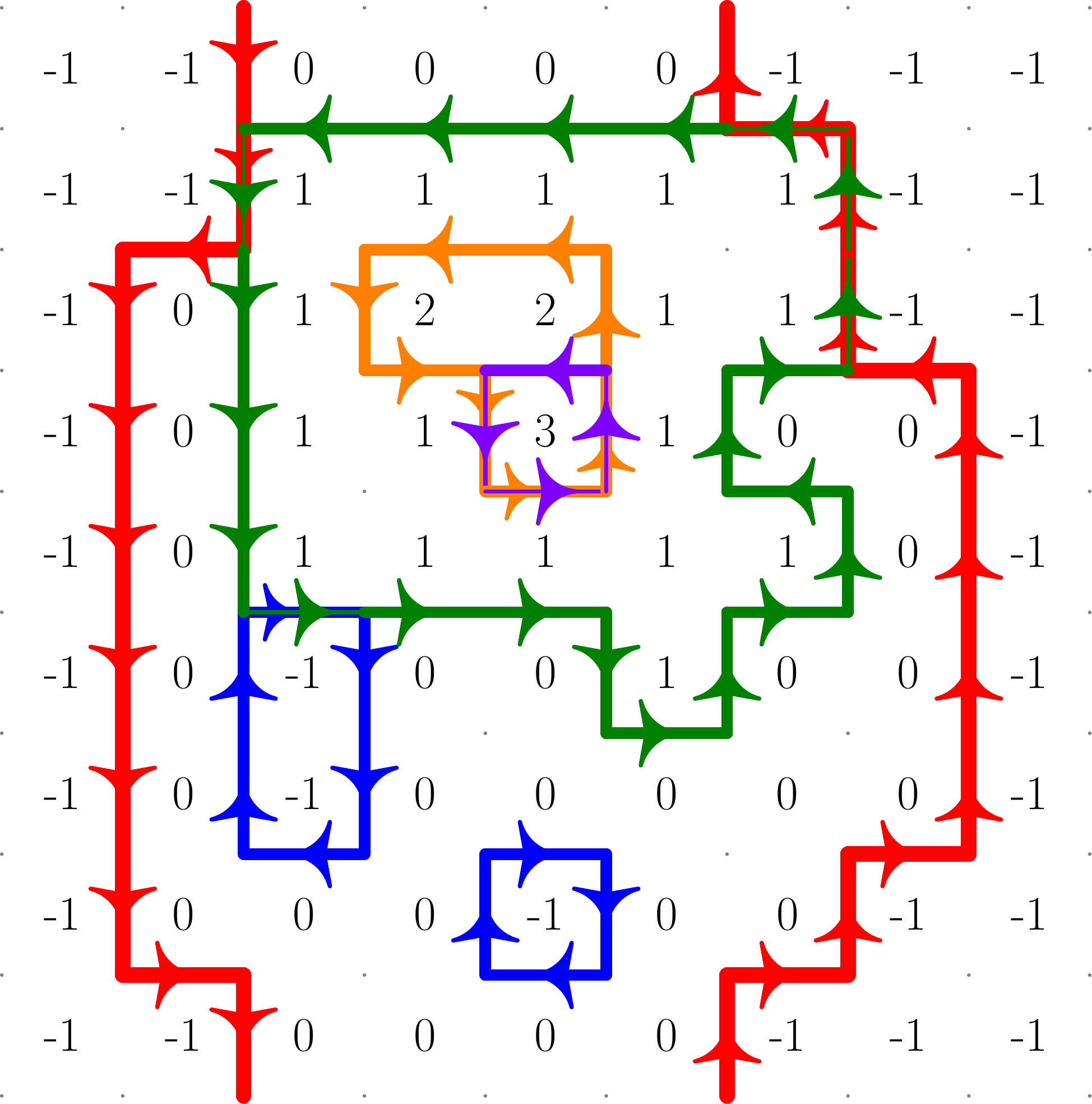}
\caption{Possible spin and contour configuration. Here, case 1 occurs for $x_i=1$ and $x_j=0$ with the green contour $\gamma_{i,j}$, while case 2 occurs for $x_i=0$ and $j$ in the left lower corner with $\gamma_{i,j}$ the union of the red contours.}\label{figure}
\end{figure}

Define $\Gamma_{i,j}:=\{\gamma_{i,j}(x): x\in\Omega_{i,j}\}$. The length $|\gamma_{i,j}|$ of a contour/pair of contours $\gamma_{i,j}\in\Gamma_{i,j}$ is defined to be its total number of edges. 
\begin{lemma}[Peierls argument] Let $\beta>0$ and $i,j\in\Lambda$ with $i\neq j$. Then for any $x\in \Omega_{i,j}$ one has
\begin{equation} \label{eq:hamiltonian_contour_inequality}
H_{\Lambda}(x)-H_{\Lambda}(F_{i,j}(x))\geq |\gamma_{i,j}(x)|. 
\end{equation}
Further, for any $\gamma\in\Gamma_{i,j}$ and $k\in \N$, the following estimate holds
\begin{align}\label{eq:comparison_F_gamma}
P_{\Lambda,\beta}(x_i-x_j\geq k, \gamma_{i,j}(x)=\gamma)\leq \ee^{-\beta|\gamma|}P_{\Lambda,\beta}(x_i-x_j\geq k-1).
\end{align}
In particular,
\begin{equation}\label{eq:comparison_F_gamma_sum}
P_{\Lambda,\beta}(x_i-x_j\geq k)\leq P_{\Lambda,\beta}(x_i-x_j\geq k-1) \sum_{\gamma\in\Gamma_{i,j}} \ee^{-\beta|\gamma|}.
\end{equation}
Hence, we have the inequality
\begin{equation}\label{eq:comparison_F_gamma_iterated}
P_{\Lambda,\beta}(x_i-x_j\geq k)\leq \bigg(\sum_{\gamma\in\Gamma_{i,j}} \ee^{-\beta|\gamma|}\bigg)^k.
\end{equation}
\end{lemma}

\begin{proof}
Given an edge $\{l,m\}\in\cE_\Lambda$ and $x\in\Omega_{i,j}$, formula \eqref{eq:F_ij} implies 
\begin{align}
(x_l -x_m)^2&=(|F_{i,j}(x)_l-F_{i,j}(x)_m|+\textbf{1}_{\{\{l,m\}^*\in\gamma_{i,j}(x)\}})^2 \nonumber \\
&\geq(F_{i,j}(x)_l-F_{i,j}(x)_m)^2+\textbf{1}_{\{\{l,m\}^*\in\gamma_{i,j}(x)\}} . \label{eq:contour_ineq}
\end{align}
Summing over $\{l,m\}\in\cE_\Lambda$ in \eqref{eq:contour_ineq} yields  \eqref{eq:hamiltonian_contour_inequality}. 

Let $\gamma\in\Gamma_{i,j}$. Note that $F_{i,j}$ restricted to $\{x_i-x_j\geq k, \gamma_{i,j}(x)=\gamma\}$ maps one-to-one into $\{x_i-x_j\geq k-1\}$ with the inverse $x\mapsto (x_l+\mathbf{1}_{\cL_{i}(\gamma)}(l)-x_0-\textbf{1}_{\cL_{i}(\gamma)}(0))_{l\in\Lambda}$.
By plugging formula \eqref{eq:hamiltonian_contour_inequality} into the definition of $P_{\Lambda,\beta}$ we obtain claim \eqref{eq:comparison_F_gamma} as follows:
\begin{align}
&P_{\Lambda,\beta}(x_i-x_j\geq k, \gamma_{i,j}(x)=\gamma)=\frac{1}{Z_{\Lambda,\beta}}\sum_{\substack{x:\; x_i-x_j\geq k, \\ \gamma_{i,j}(x)=\gamma}} \ee^{-\beta H_\Lambda(x)}\leq\frac{1}{Z_{\Lambda,\beta}}\sum_{\substack{x:\; x_i-x_j\geq k,\\ \gamma_{i,j}(x)=\gamma}} \ee^{-\beta|\gamma|} \ee^{-\beta H_\Lambda(F_{i,j}(x))} \nonumber\\
&=\ee^{-\beta|\gamma|}P_{\Lambda,\beta}(F_{i,j}(\{x_i-x_j\geq k,\gamma_{i,j}(x)=\gamma\}))\leq \ee^{-\beta|\gamma|}P_{\Lambda,\beta}(x_i-x_j\geq k-1) .
\end{align}
Summing this over $\gamma\in\Gamma_{i,j}$ yields \eqref{eq:comparison_F_gamma_sum}. 
The last claim \eqref{eq:comparison_F_gamma_iterated} follows by iterating \eqref{eq:comparison_F_gamma_sum}.
\end{proof}

In the following, $\Lambda\Subset \Z^2$ means that $\Lambda$ is an $N\times N$ box  of length $N\geq 4$ with periodic boundary conditions containing the origin.
\begin{proposition}[Counting contours]
Given $\Lambda\Subset \Z^2$, let $i,j\in\Lambda$ with $i\neq j$ and $\ell\in\N$. The number of contours/pairs of contours in $\Gamma_{i,j}$ of length $\ell$ is bounded by
\begin{align}
|\{\gamma\in\Gamma_{i,j}: |\gamma|=\ell\}|\leq 3\ell^23^{\ell}.
\end{align}
For all $\beta\geq \log 6$ and $k\in\N$ we have
\begin{equation}\label{eq:Peierls_gen}
\sup_{\Lambda\Subset\Z^2}\sup_{i,j\in\Lambda} P_{\Lambda,\beta}(|x_i - x_j|\geq k)\leq 2\varphi(\beta)^k,
\end{equation}
where $\varphi(\beta):=480(3\ee^{-\beta})^4$. 

\end{proposition}
\begin{proof}

The number of connected cycles (Case 1 above) is, as in the Ising model case (see \cite{friedli-velenik2018}), bounded by $\frac23\ell 3^{\ell}$, for all $\ell\geq 4$. We note that the number of winding pairs $\gamma_{+}\cup\gamma_{-}$ (Case 2 above) is non-null only if their total length $\ell$ is larger or equal to twice the box length $N$ (at least one $N$ for each contour $\gamma_\pm$). Moreover, it is bounded by $(2N)^2\cdot 4^2\cdot 3^{\ell-2}$. One can see this as follows: each contour $\gamma_{+}$ and $\gamma_{-}$ meets the $x-$axis or $y-$axis on the torus in at least one point ($\leq(2N)^2$ choices). Starting from this point, there are four choices for the next step ($4^2$ choices together). For the remaining $\ell-2$ steps, starting with $\gamma_{+}$, there are at most $3$ choices per step, finishing $\gamma_{+}$ and starting with $\gamma_{-}$ as soon as the path becomes closed. Altogether, this means 
\begin{equation}
|\{\gamma\in\Gamma_{i,j}: |\gamma|=\ell\}|\leq\frac23\ell 3^{\ell}+\mathbf{1}_{\{\ell\geq 2N\}} 64 N^2 3^{\ell-2}\leq 3\ell^23^{\ell}.
\end{equation}
Take $\beta\geq \log 6$, which means $3\ee^{-\beta}\leq \frac12$. Summing over $\ell\geq 4$, one obtains
\begin{equation}
\sum_{\gamma\in\Gamma_{i,j}} \ee^{-\beta|\gamma|}\leq \sum_{\ell\geq 4}\ee^{-\beta\ell}|\{\gamma\in\Gamma_{i,j}: |\gamma|=\ell\}|\leq 3\sum_{\ell\geq 4} \ell^2(3\ee^{-\beta})^\ell  \leq 480(3\ee^{-\beta})^4,
\end{equation}
where we have used $\sum_{\ell=4}^\infty 3 \ell^2 z^\ell = 3 z^4 (9 z^2 - 23 z + 16)(1-z)^{-3}\leq 480 z^4$ for $0\leq z\leq 1/2$. We haven't tried to optimize the constant.

Using first that $P_{\Lambda,\beta}$ is symmetric with respect to the reflection $x\mapsto -x$, and then inequality~\eqref{eq:comparison_F_gamma_iterated} we conclude for $k\in\N$
\begin{align}
\sup_{\substack{\Lambda\Subset\Z^2\\N\geq 4}}\sup_{i,j\in\Lambda} P_{\Lambda,\beta}(|x_i - x_j|\geq k)&\leq 2 \sup_{\substack{\Lambda\Subset\Z^2\\N\geq 4}}\sup_{i,j\in\Lambda} P_{\Lambda,\beta}(x_i - x_j\geq k)\nonumber \\ 
&\leq 2\sup_{\substack{\Lambda\Subset\Z^2\\N\geq 4}}\sup_{i,j\in\Lambda}\bigg(\sum_{\gamma\in\Gamma_{i,j}} \ee^{-\beta|\gamma|}\bigg)^k\leq 2\varphi(\beta)^k.
\end{align}
\end{proof}

For $i,j\in\Lambda$ we define the observable $\obs_{i,j}\colon\Omega_\Lambda^\R\to\R$, 
\begin{align}
\obs_{i,j}(x)=(x_i-x_j)^2. 
\end{align}

\begin{corollary}[Spontaneous breaking of the internal $\Z^1$ symmetry] 
\label{cor:symmetry_breaking}
For all $\beta\geq 3$ one has 
\begin{equation}\label{eq:bound_square_mean}
\sup_{\Lambda\Subset\Z^2}\sup_{i,j\in\Lambda} E_{\Lambda,\beta}[\obs_{i,j}(x)]\le M_\beta:=\frac{2\varphi(\beta)(1+\varphi(\beta))}{(1-\varphi(\beta))^3}.
\end{equation}
The bound $M_\beta$ is asymptotically equivalent to $2\varphi(\beta)=960(3\ee^{-\beta})^4$ as $\beta\to\infty$. In particular,
\begin{equation}\label{eq:Mbeta-asymptotics}
\lim_{\beta\to\infty} \beta M_\beta=0.
\end{equation}
\end{corollary}
\begin{proof} Let $\Lambda\Subset\Z^2$ and $i,j\in\Lambda$. For all $\beta\geq 3$ one has $\varphi(\beta)<1$ and we obtain
\begin{align}
0\leq E_{\Lambda,\beta}[\obs_{i,j}(x)]=&\sum_{n\in\N} n^2 P_{\Lambda,\beta}(|x_i-x_j|=n)\nonumber\\
\leq&\sum_{n\in\N} n^2 P_{\Lambda,\beta}(|x_i-x_j|\geq n)\nonumber\\
\leq&2\sum_{n\in\N} n^2 \varphi(\beta)^n=\frac{2\varphi(\beta)(1+\varphi(\beta))}{(1-\varphi(\beta))^3}=M_\beta<\infty. \label{eq:Upper_bound_beta}
\end{align}
Since $\varphi$ neither depends on $i,j$, nor on $\Lambda$, the above estimates 
imply the desired conclusion ~\eqref{eq:bound_square_mean}.
\end{proof}

\section{Main result}
For $i,j\in\Lambda$ define the observable $U_{ij}:\R^\Lambda\to\R$ by
\begin{align}
U_{ij}(k):=\sum_{\ell\in\Lambda} (G_{i\ell}-G_{j\ell})k_\ell.
\end{align}
We interpret $\sum_{\ell\in\Lambda} G_{i\ell}k_\ell$ as the electric potential at location $i$ of a charge distribution encoded by $(k_\ell)_{\ell\in\Lambda}$. In this interpretation, $U_{ij}$ encodes the voltage between $i$ and $j$. We remark that the thermodynamic limit of the potential kernel $G_{ii}-G_{ij}$ is well-defined cf. Theorem 1.6.1 in \cite{MR2985195}. Since $G_{ii}-G_{ij}\sim C\log|i-j|$ in the double limit $\Lambda\nearrow\Z^2$ and then $|i-j|\to\infty$, the following result states that the variance of the energy cost of transporting a unit charge from $i$ to $j$ asymptotically behaves like $c(\beta^*)^{-1}\log|i-j|$, for some positive constants $C,c$.

\begin{theorem}[Variance of voltages in a lattice Coulomb gas]
\label{thm:main}  For all $N\times N$ boxes $\Lambda\ni 0$  of length $N\geq 4$ with periodic boundary conditions, all $i,j\in\Lambda$, and all $\beta^*\leq \frac{1}{12}$ one has
\begin{equation}\label{eq:energy-cost}
\frac{4}{\beta^*}(G_{ii}-G_{ij})-\frac{4}{(\beta^*)^2} M_{(4\beta^{*})^{-1}}\leq \var^*_{\Lambda,\beta^*}(U_{ij})= E^*_{\Lambda,\beta^*}[U_{ij}^2]\leq \frac{4}{\beta^*}(G_{ii}-G_{ij}),
\end{equation}
where $E^*_{\Lambda,\beta^*}$ and $\var^*_{\Lambda,\beta^*}$ denote the expectation and the variance with respect to $P_{\Lambda,\beta^*}^*$, respectively. As a consequence, we obtain the following asymptotic equivalence
\begin{equation}
E^*_{\Lambda,\beta^*}[U_{ij}^2]\sim \frac{4}{\beta^*}(G_{ii}-G_{ij}) \; \text{ as } \; \beta^*\to0.
\end{equation} 
\end{theorem}

\begin{proof}
Given $\Lambda$, let $\md x:=\delta_0(\md x_0)\prod_{\ell\in\Lambda\setminus\{0\}}\md x_\ell$ denote the Lebesgue measure on $\Omega^\R_\Lambda$. Set $\beta:=(4\beta^*)^{-1}\geq 3$. Given $i,j\in\Lambda$ let $h(x)=\obs_{i,j}(x)\ee^{-\beta H_\Lambda(x)}$. Using formula \eqref{eq:Fourier-g} with 
$n=|\Lambda|-1$ and $A=-\beta \Delta_{0^c0^c}$ we obtain its Fourier transform at 
$k\in\R^\Lambda$ 
\begin{align}
\hat h(k)= &\int_{\Omega^\R_\Lambda} \ee^{-\ii k\cdot x}\obs_{i,j}(x)\ee^{-\beta 
H_\Lambda(x)}\md x
= -(\partial_{k_i}-\partial_{k_j})^2
\int_{\Omega^\R_\Lambda} \ee^{-\ii k\cdot x} \ee^{-\beta H_\Lambda(x)}\md x \nonumber\\
= & -\left(\frac{\pi}{\beta}\right)^{\frac{|\Lambda|-1}{2}}
\det(-\Delta_{0^c0^c})^{-\frac12}(\partial_{k_i}-\partial_{k_j})^2
\ee^{\beta^* k^t_{0^c}(\Delta_{0^c0^c})^{-1}k_{0^c}}.
\end{align}
Applying the Poisson summation formula \eqref{eq:PoissonSummationFormula_simple} again, we obtain that
\begin{align}\label{eq:Poisson}
\sum_{x\in\Omega_\Lambda}\obs_{i,j}(x) \ee^{-\beta H_\Lambda(x)}= -\left(\frac{\pi}{\beta}\right)^{\frac{|\Lambda|-1}{2}}
\det(-\Delta_{0^c0^c})^{-\frac12}\sum_{k\in \OLstar}(\partial_{k_i}-\partial_{k_j})^2
\ee^{\beta^* k^t_{0^c}(\Delta_{0^c0^c})^{-1}k_{0^c}}.
\end{align}
According to Corollary \ref{cor:rel-G-delta-inverse}, the expression
$\ee^{\beta^* k^t_{0^c}(\Delta_{0^c0^c})^{-1}k_{0^c}}$ coincides 
with $\ee^{-\frac{\beta^*}{4} k^tGk}$ for $k\in\K_\Lambda$. Consequently, all 
directional derivatives of these two quantities 
in directions of the hyperplane ~$\K_\Lambda$ coincide as well. In particular, all 
$\partial_{k_i}-\partial_{k_j}$ are such directional derivatives. 
Using this in \eqref{eq:Poisson} and plugging in the formula for the partition sum \eqref{eq:Z-poisson}, we calculate 
\begin{align}
E_{\Lambda,\beta}[\obs_{i,j}]=&\frac{1}{Z_{\Lambda,\beta}}\sum_{x\in\Omega_\Lambda}
\obs_{i,j}(x) \ee^{-\beta H_\Lambda(x)}\nonumber\\
= &-\frac{1}{Z_{\Lambda,\beta^*}^*}\sum_{k\in \OLstar}(\partial_{k_i}-\partial_{k_j})^2
\ee^{\beta^* k^t_{0^c}(\Delta_{0^c0^c})^{-1}k_{0^c}}\nonumber\\
= &-\frac{1}{Z_{\Lambda,\beta^*}^*}\sum_{k\in \OLstar}(\partial_{k_i}-\partial_{k_j})^2
\ee^{-\frac{\beta^*}{4} k^tGk}\label{eq:expectation_obs_1}
. 
\end{align}
Moreover, by translation invariance ($G_{ii}=G_{jj}$) and symmetry ($G_{ij}=G_{ji}$) of $G$ it follows 
\begin{align}
&(\partial_{k_i}-\partial_{k_j})^2\ee^{-\frac{\beta^*}{4} k^tGk}\nonumber\\ 
&=\frac{\beta^*}{4} \ee^{-\frac{\beta^*}{4} k^tGk}\bigg[\frac{\beta^*}{4}
\Big(\sum_{\ell\in\Lambda} (G_{i\ell}+G_{\ell i}-G_{j\ell}-G_{\ell j})k_\ell\Big)^2 - 2\big(G_{ii}+G_{jj}-G_{ij}-
G_{ji}\big)\bigg]\nonumber\\
&= \ee^{-\frac{\beta^*}{4} k^tGk}\Big(\frac{(\beta^*)^2}{4} U_{ij}^2(k) 
 - \beta^*(G_{ii}-G_{ij})\Big).\label{eq:second_derivative}
\end{align}
Putting \eqref{eq:bound_square_mean}, \eqref{eq:expectation_obs_1} and 
\eqref{eq:second_derivative} together yields
\begin{align}
0\leq -\frac{1}{Z_{\Lambda,\beta^*}^*}\sum_{k\in \OLstar}\ee^{-\frac{\beta^*}{4} k^tGk}\Big(\frac{(\beta^*)^2}{4} U_{ij}^2(k) 
 - \beta^*(G_{ii}-G_{ij})\Big)\leq M_\beta=M_{(4\beta^*)^{-1}},
\end{align}
or equivalently,
\begin{align}
0\geq E^*_{\Lambda,\beta^*}[U_{ij}^2] -\frac{4}{\beta^*}(G_{ii}-G_{ij})\geq -\frac{4}{(\beta^*)^2} M_{(4\beta^{*})^{-1}}.
\end{align}
This implies the desired conclusion \eqref{eq:energy-cost}, taking into account that $E^*_{\Lambda,\beta^*}[U_{ij}]=0$ by translation invariance.
\end{proof}

\section{Discussion} 
In this section we compare the main result of this paper with results on Debye screening in three dimensions by Brydges \cite{B} and Brydges-Federbush \cite{BF} and results on two-dimensional abelian spin systems and the Coulomb gas by Fr\"ohlich and Spencer \cite{FS}.

As is mentioned in \cite{FS}, their Theorem A in Section 1.3 also applies for the discrete Gaussian model in the special case of weights equal to $1$ on the integers. More precisely, it applies to the \emph{large} temperature regime in the discrete Gaussian model, because the sine-Gordon transformation used there inverts the temperature, so that their $\beta$ plays the role of the temperature. The lower bound for the high-temperature regime in formula (1.20) in \cite{FS} should be seen in contrast to our bound \eqref{eq:bound_square_mean} in the low-temperature regime. The former provides a lower bound which increases logarithmically with distance, while the latter states an upper bound which is uniform in the distance. This illustrates a difference between the high- and low-temperature phases in the two-dimensional discrete Gaussian model. Fr\"ohlich and Spencer \cite{FS} use another way of dualizing the discrete Gaussian model, obtaining the Villain model. In contrast to the discrete Coulomb gas, it deals with continuous angle variables and thus their Theorems C-E cannot be directly compared to our results.

As was mentioned in the introduction of \cite{FS}, ``the Coulomb gas has a high temperature, low density plasma phase characterized by exponential Debye screening''. This was examined in three space dimensions in several papers, in particular, by Brydges in \cite{B} and Brydges-Federbush in \cite{BF}. More precisely, they prove a decay of correlations for observables depending only on the charge particle configuration on compact regions, which is exponential in the distance between these regions. The voltage $U_{ij}$ examined in this paper is not an observable in this class, because it does not only depend on the charge configuration close to $i$ and $j$, but on the whole charge configuration. In this sense, it is not a local observable. Theorem \ref{thm:main} above shows that in the high temperature regime of the two-dimensional lattice Coulomb gas the variance of $U_{ij}$ increases logarithmically in the distance $|i-j|$. We find this an interesting observation, which complements the previously mentioned results on Debye screening. In view of phase transitions, it remains an interesting open problem to study the voltage $U_{ij}$ also in the low-temperature regime.

\paragraph{Acknowledgment.} We thank Aernout van Enter and anonymous referees for remarks on the literature. We also thank the referees for their constructive comments, which helped us to improve the paper.


\end{document}